\documentclass[11pt]{article}
 
\usepackage[margin=1in]{geometry} 
\usepackage{titling}
\usepackage{amsmath,amsthm,amssymb}
\usepackage{mathtools}
\usepackage{color}
\usepackage{lipsum} 
\usepackage{fancyhdr}
\usepackage{tikz}
\usetikzlibrary{patterns}
\usepackage{ytableau}
\usepackage{physics}
\usepackage{enumerate}
\usepackage{caption}
\newcommand{\Z}{\mathbb{Z}}
\newcommand{\Q}{\mathbb{Q}}

\newcommand{\C}{\mathbb{C}}
\newcommand{\uq}{U_q^{\text{tw}}(\mathfrak{so}_3)}
\newcommand{\uqq}{U_q^{\text{tw}}(\mathfrak{so}_4)}
\newcommand{\uqqq}{U_q^{\text{tw}}(\mathfrak{so}_n)}

\newcommand{\inv}{^{-1}}
\newcommand{\f}{\Tilde{f}}

\predate{}
\postdate{}

\author{Jordan Disch}
\title{Generic Gelfand-Tsetlin Representations of $\uq$ and $\uqq$}
\date{}
\begin{document}

\newtheorem{theorem}{Theorem}[section]
\newtheorem{lemma}[theorem]{Lemma}
\newtheorem{corollary}{Corollary}[theorem]
\newtheorem{proposition}[theorem]{Proposition}
\newtheorem{case}{Case}
\newtheorem{definition}[theorem]{Definition}
\newtheorem{example}[theorem]{Example}

\maketitle

\begin{abstract}
    \noindent We construct generic Gelfand-Tsetlin representations of the $\imath$quantum groups $\uq$ and $\uqq$. These representations are infinite-dimensional analogs to the finite-dimensional irreducible representations provided by Gavrilik and Klimyk in \cite{gav}. They are quantum analogs of generic Gelfand-Tsetlin representations constructed by Mazorchuk in \cite{ma}. We give sufficient conditions for irreducibility and provide an upper bound for the length with the help of Casimir elements found by Molev, Ragoucy, and Sorba.
\end{abstract}

\section{Introduction}
Fix $h\in\C\setminus2\pi i\Q$ and put $q=\exp(h)$, where $i=\sqrt{-1}$. Thus $q$ is not a root of unity. In \cite{gav}, Gavrilik and Klimyk gave finite-dimensional irreducible representations of the algebra $\uqqq$, a coideal subalgebra of the quantum group $U_q(\mathfrak{gl}_n)$, as discussed in \cite{mol}. These algebras are realized as a quantum symmetric pair in \cite{nou}, which are discussed in detail in \cite{let}. More generally, the algebra $\uqqq$ is an example of an $\imath$quantum group, which are discussed in detail in \cite{wang}. The representations from \cite{gav} are parameterized by the Gelfand-Tsetlin patterns given by \cite{gz}. Following the work started in \cite{dfo}, and arguing in the spirit of \cite{maz}, our goal is to construct infinite-dimensional analogues to the Gavrilik-Klimyk representations from \cite{gav}. We call them generic Gelfand-Tsetlin representations, following the naming conventions given in \cite{ga}. We define these generic Gelfand-Tsetlin representations for $n=3$ and $n=4$.  We prove that they exist and show that under certain restrictions, they have finite length. We also provide an upper bound for the length, along with some examples. When $n=4$, this requires finding a nontrivial central element $C_q\in Z(\uq)$ whose eigenspaces are one-dimensional. We obtain this element from \cite{mol}, though they use a presentation of $\uqqq$ different from that of \cite{gav}. However, \cite{mo} shows the connection between the presentations. This allows us then to have $C_q$ in terms of our original presentation.

In Section \ref{prelim}, we give the necessary lemmas for the following sections. In Section \ref{uq}, we define and prove the existence of our generic Gelfand-Tsetlin representations of $\uq$. We show that under the right conditions, they have a length of at most $3$ and provide sufficient conditions for irreducibility, along with some examples. We also define the central element $C_q$ and give its eigenvalue on the basis vectors. In Section \ref{uqq}, we define and prove the existence of our generic Gelfand-Tsetlin representations of $\uqq$. Under the right conditions, they have a length of at most $6$, and we also provide sufficient conditions for irreducibility and examples.

\section{Preliminaries}\label{prelim}
In this section, we define our algebra $\uqqq$ and provide certain finite-dimensional irreducible representations. We will use these to construct infinite-dimensional analogs for $\uq$ and $\uqq$ in Sections \ref{uq} and \ref{uqq}. We introduce $q$-numbers and prove useful facts about them for later. We prove that a certain set is Zariski-dense in complex affine space, a fact that will be critical for constructing generic Gelfand-Tsetlin representations.

\subsection{$\uqqq$ and Finite-Dimensional Irreducible Representations}
Here we define $\uqqq$ for $n\geq 2$ using the same presentation as \cite[Section~2]{gav}, and give the finite-dimensional representations for $\uq$ and $\uqq$ given in \cite[Propositions~1,2]{ga}. ($U_q^{\text{tw}}(\mathfrak{so}_2)$ is just the polynomial algebra in one generator $\C[I_{21}]$.)

For all $b\in\C$, we define $q^b=\exp(hb)$ and
\begin{equation*}
    [b]:=\dfrac{q^{b}-q^{-b}}{q-q^{-1}}.
\end{equation*}
$\uqqq$ is defined as the complex associative algebra generated by elements $I_{i,i-1}$, $i=2,...,n$, which satisfy the following relations:
\begin{subequations}
\begin{equation}\label{rel1}
    [I_{i,i-1},I_{j,j-1}]=0 \ \ \ \ \ \text{if} \ |i-j|>1,
    \end{equation}
    \begin{equation}\label{rel2}
    I_{i+1,i}^2I_{i,i-1}-[2]I_{i+1,i}I_{i,i-1}I_{i+1,i}+I_{i,i-1}I_{i+1,i}^2=-I_{i,i-1},
    \end{equation}
    \begin{equation}\label{rel3}
    I_{i,i-1}^2I_{i+1,i}-[2]I_{i,i-1}I_{i+1,i}I_{i,i-1}+I_{i+1,i}I_{i,i-1}^2=-I_{i+1,i}.
    \end{equation}
\end{subequations}
    As $q\to 1$, this becomes the universal enveloping algebra of the complex semisimple Lie algebra $\mathfrak{so}_n$. This is because \eqref{rel2} and \eqref{rel3} become
\begin{equation*}
        [I_{i+1,i},[I_{i+1,i},I_{i,i-1}]]=-I_{i,i-1},
    \end{equation*}
    \begin{equation*}
        [I_{i,i-1},[I_{i,i-1},I_{i+1,i}]]=-I_{i+1,i},
    \end{equation*}respectively.
    
Now we give the finite-dimensional irreducible representations defined in \cite[Propositions~1,2]{ga}. When $n=3$, the irreducible representation $V_{\ell}$ is characterized by highest weight $\ell\in \frac{1}{2}\Z_{\geq 0}$. Then $V_{\ell}$ has basis $\{\ket{m}\mid \ell\geq m \geq -\ell, \ \ell-m \in \Z\}$, and the action is as follows:
\begin{subequations}
    \begin{equation} \label{i213}
        I_{21}.\ket{m}=i[m]\ket{m},
    \end{equation}
    \begin{equation} \label{i323}
        I_{32}.\ket{m}=A_{\ell,m}\ket{m+1}-A_{\ell,m-1}\ket{m-1},
    \end{equation}
    \end{subequations}
    where $\ket{-\ell-1}=0$, $\ket{\ell+1}=0$, and
    \begin{equation*}
        A_{\ell,m}=\bigg(\dfrac{[m][m+1]}{[2m][2m+2]}\cdot [\ell+m+1][\ell-m]\bigg)^{\frac{1}{2}}.
    \end{equation*}When $n=4$, the irreducible representation $V_{p,r}$ is characterized by highest weight $(p,r)$ where $p,r\in\frac{1}{2}\Z$, $p-r\in\Z$, and $p\geq |r|$. The basis is given by $\mathcal{B}:=\{\ket{\ell,m}\mid p\geq \ell\geq|r|, \ \ell\geq m\geq -\ell, \ p-\ell\in\Z, \ p-m\in\Z\}$, and the action of $\uqq$ on $V_{p,r}$ is as follows:
    \begin{subequations}
    \begin{equation}\label{i21}
        I_{21}.\ket{\ell,m}=i[m]\ket{\ell,m},
    \end{equation}
    \begin{equation} \label{i32}
        I_{32}.\ket{\ell,m}=A_{\ell,m}\ket{\ell,m+1}-A_{\ell,m-1}\ket{\ell,m-1},
    \end{equation}
    \begin{equation} \label{i43}
        I_{43}.\ket{\ell,m}=B_{\ell,m}\ket{\ell+1,m}-B_{\ell-1,m}\ket{\ell-1,m}+iC_{\ell,m}\ket{\ell,m}
    \end{equation}
    \end{subequations}
    where
    \begin{subequations}
    \begin{equation} \label{alm}
        A_{\ell,m}=\bigg(\dfrac{[m][m+1]}{[2m][2m+2]}\cdot [\ell+m+1][\ell-m]\bigg)^{\frac{1}{2}},
    \end{equation}
    \begin{equation} \label{blm}
    B_{\ell,m}=\bigg(\dfrac{[p+\ell+2][p-\ell][\ell+r+1][\ell-r+1][\ell+m+1][\ell-m+1]}{[\ell+1]^2[2\ell+1][2\ell+3]}\bigg)^{\frac{1}{2}},
\end{equation}
\begin{equation} \label{clm}
    C_{\ell,m}=\dfrac{[p+1][r][m]}{[\ell+1][\ell]},
\end{equation}
\end{subequations}and we define $\ket{\ell,m}=0$ if $\ell$ and $m$ fail to satisfy the interlacing conditions given in $\mathcal{B}$.

\subsection{$q$-Numbers}
The following lemma contains useful results about $q$-numbers which will be used in Sections \ref{uq} and \ref{uqq}. We use the fact that $[b]=0$ if and only if $q^{2b}=1$.
\begin{lemma}\label{lemq}
The following are true:
\begin{enumerate}[{\rm (i)}]
    \item\label{lemq1}If $b\in \Q\setminus\{0\}$, then $[b]\neq 0$.
    \item\label{lemq1.5}If $a-b\in\frac{1}{2}\Z\setminus\{0\}$, then $q^{2a}\neq q^{2b}$.
    \item\label{lemq2}If $[c]=0$, then $[c+k]\neq 0$ for all $k\in\frac{1}{2}\Z\setminus\{0\}$.
    \item\label{lemq3}$\frac{[m]}{[2m]}$ exists if and only if $q^{2m}\neq -1$.
    \item\label{lemq4}$q^{2m_0+k}\neq -1$ for all $k\in \Z$ if and only if $[m_1]\neq [m_2]$ for all $m_1,m_2\in m_0+\Z$ such that $m_1\neq m_2$.
    \item\label{lemq5}$q^{4\ell_0+2k}\neq 1$ for all $k\in\Z$ if and only if $[\ell_1]^2+q^{\ell_1+1}[\ell_1]\neq[\ell_2]^2+q^{\ell_2+1}[\ell_2]$ for all $\ell_1,\ell_2\in\ell_0+\Z$ such that $\ell_1\neq\ell_2$.
\end{enumerate}
\end{lemma}
\begin{proof}
Suppose $[b]=0$ where $b=\frac{r}{s}$ and $r,s\in\Z$. Without loss of generality, assume $s>0$. Then
\begin{equation*}
    q^{2b}=1
\end{equation*}which implies
\begin{equation*}
    (q^{2b})^s=q^{2r}=1=q^{-2r}.
\end{equation*}Since $r\neq 0$, either $r\in\Z_{>0}$ or $-r\in\Z_{>0}$. But $q$ is not a root of unity, so this is a contradiction.

Now we prove \eqref{lemq1.5}. Suppose $a\neq b$ and $a-b\in\frac{1}{2}\Z$. We show by contradiction that $q^{2a}\neq q^{2b}$. If $q^{2a}=q^{2b}$, then $q^{2(a-b)}=1$, which contradicts the fact that $q$ is not a root of unity.

The proof for \eqref{lemq2} follows from \eqref{lemq1.5}, setting $a=c+k$ and $b=c$, and the fact that a $q$-number $[c]$ is zero if and only if $q^{2c}=1$.

To prove \eqref{lemq3}, note
\begin{equation*}
    \dfrac{[m]}{[2m]}=\dfrac{q^{m}-q^{-m}}{q^{2m}-q^{-2m}}=\dfrac{1}{q^{m}+q^{-m}}
\end{equation*}so we get existence if and only if $q^{m}\neq -q^{-m}$, which is equivalent to $q^{2m}\neq -1$.

Now we prove \eqref{lemq4}. Suppose $[m_1]=[m_2]$ where $m_1,m_2\in m_0+\Z$ and $m_1\neq m_2$. Then setting $x:=q^{m_1}$ and $y:=q^{m_2}$, it is straightforward to check the following equations are equivalent:
\begin{subequations}
\begin{equation*}
    x-\frac{1}{x}=y-\frac{1}{y}
\end{equation*}and
\begin{equation*}
    (xy+1)(x-y)=0.
\end{equation*}This is equivalent to
\begin{equation*}
    y=-\frac{1}{x}.
\end{equation*}(Suppose $x=y$. Then $x^2=y^2$, which contradicts \eqref{lemq1.5}.) So
\begin{equation*}
    q^{m_2}=-q^{-m_1}
\end{equation*}which is equivalent to
\begin{equation*}
    q^{m_1+m_2}=-1.
\end{equation*}Equivalently,
\begin{equation*}
     q^{2m_0+k}=-1 \ \text{for some $k\in \Z$}.
\end{equation*}
\end{subequations}

Finally, we prove \eqref{lemq5}. It is straightforward to check
\begin{equation*}
    [\ell]^2+q^{\ell+1}[\ell] =\frac{q^{2\ell+2}+q^{-2\ell}-q^2-1}{(q-q\inv)^2}.
\end{equation*}So for $\ell_1,\ell_2\in\ell_0+\Z$ where $\ell_1\neq \ell_2$, we have
\begin{equation*}
    [\ell_1]^2+q^{\ell_1+1}[\ell_1]=[\ell_2]^2+q^{\ell_2+1}[\ell_2]
\end{equation*}is equivalent to
\begin{equation*}
    q^{2\ell_1+2}+q^{-2\ell_1}=q^{2\ell_2+2}+q^{-2\ell_2}.
\end{equation*}Let $x:=q^{2\ell_1}$ and $y:=q^{2\ell_2}$. It is straightforward to check the following equations are equivalent:
\begin{subequations}
\begin{equation*}
    q^2x+x\inv=q^2y+y\inv
\end{equation*}and
\begin{equation*}
    (q^2xy-1)(x-y)=0.
\end{equation*}This is equivalent to
\begin{equation*}
    q^2xy=1.
\end{equation*}(Note $x\neq y$ by \eqref{lemq1.5}.) Then this is equivalent to
\begin{equation*}
     q^{4\ell_0+2k_1+2k_2+2}=1
\end{equation*}
\end{subequations}
for some $k_1,k_2\in\Z$, which implies \eqref{lemq5}.
\end{proof}

\subsection{Zariski-Density}
The lemmas in this subsection show that a certain set is Zariski-dense in complex affine space. It is necessary to define Gelfand-Tsetlin patterns for $\uqqq$ for the second, more general lemma. We use these lemmas to prove the existence of our generic Gelfand-Tsetlin representations in Theorems \ref{exist3} and \ref{exist4}. 

Lemma \ref{zar3} is used for $n=3$, while Lemma \ref{zar4} is used for general $n$. Thus Lemma \ref{zar3} is redundant since it is proven by Lemma \ref{zar4}, but we include it for the convenience of the reader since it is a simpler version.
\begin{lemma}\label{zar3} Let $S:=\{(\ell,m)\in (\frac{1}{2}\Z)^2\mid \ell\geq |m|, \ \ell- m \in \Z\}$. Let $S_q:=\{(q^{\ell},q^{m})\mid (\ell,m)\in S\}$. Suppose for $f(z,w)\in \operatorname{Frac}(\C[z,w])$, defined on $S_q$, we have $f(a,b)=0$ for all $(a,b)\in S_q$. Then $f=0$.
\end{lemma}
\begin{proof}
Note that there exists a nonzero polynomial $h(z,w)$ such that $\f(z,w):=h(z,w)f(z,w)\in \C[z,w]$. It is enough to show $\f=0$ in $\C[z,w]$. There is a nonnegative integer $N$ where for $0\leq k\leq N$ there exists $c_k(w)\in \C[w]$ such that
\begin{equation*}
    \f(z,w)=\sum_{k=0}^{N}c_k(w) z^k.
\end{equation*}Fix $m\in \frac{1}{2}\Z$. Say $b:=q^{m}$. Then
\begin{equation*}
    g_b(z):=\f(z,b)=\sum_{k=0}^{N}c_k(b)z^k\in \C[z].
\end{equation*}Note $\{q^{\ell}\mid \ell\in \frac{1}{2}\Z_{\geq |m|}, \ \ell- m \in \Z\}$ is a set of zeros for $g_b(z)$, and since $q$ is not a root of unity we have that this set is infinite. Therefore $g_b=0$ in $\C[z]$, hence $c_k(b)=0$ for all $0\leq k\leq N$ and for all $b\in \{q^{m}\mid m \in \frac{1}{2}\Z\}=:T$ since our choice of $m$ was arbitrary. Because $q$ is not a root of unity, $T$ is infinite, which means that $c_k(w)=0$ for all $0\leq k\leq N$. Therefore $\f=0$ in $\C[z,w]$, so we are done. 
\end{proof}

The following definition is needed for proving Lemma \ref{zar4}, and it is also needed for the construction of finite-dimensional representations of $\uqqq$.
\begin{definition}
For $n\geq 2$, consider a $k$-tuple $(m_{n1},m_{n2},...,m_{nk})\in\Z^k\cup(\frac{1}{2}+\Z)^k$ where $k=\operatorname{rank}\mathfrak{so}_n=\lfloor\frac{n}{2}\rfloor$ and
\begin{subequations}
    \begin{align}
        m_{n1}\geq m_{n2} \geq \dots \geq m_{nk}\geq 0, &  \ \text{for $n$ odd} \label{highest_wt_odd}\\
        m_{n1}\geq m_{n2} \geq \dots \geq |m_{nk}|, &  \ \text{for $n$ even}. \label{highest_wt_even}
    \end{align}
    \end{subequations}
    A \emph{Gelfand-Tsetlin pattern} for $U_q^{\text{tw}}(\mathfrak{so}_n)$ has the following form: 
\[
\begin{matrix}
m_{n1} & \dots & \dots & \dots & m_{nk}\\
 & & \vdots & &\\
 & m_{51} & & m_{52} &\\
 & m_{41} & & m_{42} &\\
  & & m_{31} & & \\
 & & m_{21} & &
\end{matrix}
\]where the entries satisfy the following interlacing conditions from \cite{gz}: 
\begin{subequations}
\begin{equation} \label{interlacing_odd}
    m_{2p+1,1}\geq m_{2p,1}\geq m_{2p+1,2}\geq m_{2p,2} \geq \dots \geq m_{2p+1,p}\geq m_{2p,p}\geq -m_{2p+1,p},
\end{equation}
\begin{equation} \label{interlacing_even}
    m_{2p,1}\geq m_{2p-1,1}\geq m_{2p,2}\geq m_{2p-1,2}\geq \dots \geq m_{2p-1,p-1}\geq |m_{2p,p}|.
\end{equation}
\end{subequations}
\end{definition}
The top row is the highest weight of a representation, while the lower entries parameterize the basis vectors in the following finite-dimensional representations from \cite{gav}, on which our generic Geland-Tsetlin representations are based. If $\alpha$ is a Gelfand-Tsetlin pattern which labels a basis element of our representation, 
\begin{equation*}
    I_{2p+1,2p}.\ket{\alpha}=\sum_{j=1}^pA_{2p}^j(\alpha)\ket{\alpha^j_{2p}} - \sum_{j=1}^pA_{2p}^j(\overline{\alpha}^j_{2p})\ket{\overline{\alpha}^j_{2p}}
\end{equation*}
\begin{equation*}
    I_{2p+2,2p+1}.\ket{\alpha}=\sum_{j=1}^pB_{2p}^j(\alpha)\ket{\alpha^j_{2p+1}} - \sum_{j=1}^pB_{2p}^j(\overline{\alpha}^j_{2p+1})\ket{\overline{\alpha}^j_{2p+1}}+i\frac{\prod_{r=1}^{p+1}[l_{2p+2,r}]\prod_{r=1}^p[l_{2p,r}]}{\prod_{r=1}^p[l_{2p+1,r}][l_{2p+1,r}-1]}\ket{\alpha}
\end{equation*}where $\alpha^j_i$ (correspondingly $\overline{\alpha}^j_i$) is obtained from $\alpha$ by replacing $m_{ij}$ by $m_{ij}+1$ (correspondingly, by $m_{ij}-1$); $l_{2p,j}=m_{2p,j}+p-j$, $l_{2p+1,j}=m_{2p+1,j}+p-j+1$ and
\begin{subequations}
\begin{equation}\label{a2pj}
    A_{2p}^j(\alpha)=\bigg(\dfrac{[l_j'][l_j'+1]}{[2l_j'][2l_j'+2]}\dfrac{\prod_{r=1}^p[l_r+l_j'][|l_r-l_j'-1|]\prod_{r=1}^{p-1}[l_r''+l_j'][|l_r''-l_j'-1|]}{\prod_{r\neq j}[l_r'+l_j'][|l_r'-l_j'|][l_r'+l_j'+1][|l_r'-l_j'-1|]}\bigg)^{1/2}
\end{equation}
\begin{equation}\label{b2pj}
    B_{2p}^j(\alpha)=\bigg(\dfrac{\prod_{r=1}^{p+1}[l_r+l_j'][|l_r-l_j'|]\prod_{r=1}^{p}[l_r''+l_j'][|l_r''-l_j'|]}{\prod_{r\neq j}[l_r'+l_j'][|l_r'-l_j'|][l_r'+l_j'-1][|l_r'-l_j'-1|]}\dfrac{1}{[l_j']^2[2l_j'+1][2l_j'-1]}\bigg)^{1/2}.
\end{equation}
\end{subequations}
In formula \eqref{a2pj}, $l_{2p+1,i}$ is denoted by $l_i$, $l_{2p,i}$ by $l_i'$, and $l_{2p-1,i}$ by $l_i''$. Likewise, in formula \eqref{b2pj}, $l_{2p+2,i}$ is denoted by $l_i$, $l_{2p+1,i}$ by $l_i'$, and $l_{2p,i}$ by $l_i''$.

Define $s:=r_1+\dots+r_n$, where $r_i=\operatorname{rk}\mathfrak{so}_i=\lfloor\frac{i}{2}\rfloor$. Thus $s=k^2+k$ if $n$ is odd and $s=k^2$ if $n$ is even. 
\begin{lemma}\label{zar4}
Consider a rational function $f\in \operatorname{Frac}(\C[x_1,...,x_s])$. If $f$ is zero at all $q^{m_{ij}}$ for any allowed Gelfand-Tsetlin pattern for $U_q^{\text{tw}}(\mathfrak{so}_n)$, then $f=0$. 
\end{lemma}
\begin{proof}
If $n=2k$, we show that for any element of $S_n:=\{\textbf{a}=(a_1,...,a_{s})\in \Z^{s}\cup (\frac{1}{2}+\Z)^{s}\mid a_1\geq a_2\geq \dots \geq a_{s}\geq 0\}$, there exists a valid Gelfand-Tsetlin pattern for $\uqqq$ whose entries consist of the $a_r$ for $1\leq r \leq s$. Then for $n=2k+1$ we prove the same result. Once we have done this, it will be enough to show that if $f$ is zero at all $q^{a_r}$ for all $\textbf{a}\in S_n$, then $f=0$. 

First we consider when $n$ is even, or $n=2k$. Consider $(a_1,...,a_{k^2})\in S_n$. We present the following algorithm for constructing a $\uqqq$ Gelfand-Tsetlin pattern:\\ \\ \\
Assign $m_{n1}=a_1$.\\
$j=1,...,k-1$:\\
\indent $i=1,...,j$:\\
\indent $m_{n-1-2(i-1),j-(i-1)}=a_{j^2+i}$\\
\indent end\\ \\
\indent $i=1,...,j+1$:\\
\indent $m_{n-2j+2(i-1),1+(i-1)}=a_{j^2+j+i}$\\
\indent end\\
end\\ \\ \\
(Note that the conditions for a Gelfand-Tsetlin pattern allow $m_{ij}\geq m_{i+2,j+1}$ or $m_{ij}\leq m_{i+2,j+1}$.) We provide an example for $n=8$:
\[
\begin{matrix}
a_1 & & a_4 & & a_9 & & a_{16}\\
 & a_2 & & a_5 & & a_{10} & \\
 & a_3 & & a_8 & & a_{15} & \\
 & & a_6 & & a_{11} & & \\
 & & a_7 & & a_{14} & & \\
 & & & a_{12} & & & \\
 & & & a_{13} & & & \\
\end{matrix}
\]
Now suppose $n=2k+1$. Consider $(a_1,...,a_{k^2+k})\in S_n$. Note any Gelfand-Tsetlin pattern for $U_q^{\text{tw}}(\mathfrak{so}_{n+1})$ reduces to a pattern for $\uqqq$ if you remove the highest weight, the top row consisting of $k+1$ entries. We use a modified version of the above algorithm to create a Gelfand-Tsetlin pattern for $U_q^{\text{tw}}(\mathfrak{so}_{n+1})$, which will then reduce to our desired pattern for $\uqqq$. Since a pattern for $U_q^{\text{tw}}(\mathfrak{so}_{n+1})$ has $k+1$ more entries than the number of $a_r$ that we have, we need to add $k+1$ extra assignments to the algorithm. We modify the algorithm in the following way: when we place an $a_r$ in the $m_{n+1,1},m_{n+1,2},...,m_{n+1,k}$ entries, we place it again into the next entry that the algorithm dictates, then continue along. We choose $m_{n+1,k+1}=0$. Removing the highest weight then gives us a Gelfand-Tsetlin pattern for $\uqqq$ whose entries consist of the $a_r$. Therefore for all $n\geq 2$, $S_n$ is a subset of admissible entries for a $\uqqq$ Gelfand-Tsetlin pattern (given some permutation of the $m_{ij}$; follow the algorithms).

Now suppose $f(q^{a_1},q^{a_2},...,q^{a_s})=0$ for all $\textbf{a}\in S_n$. Once we show $f=0$, we are done. As in Lemma \ref{zar3}, we may assume $f\in \C[x_1,...,x_s]$ for convenience.

There is a nonnegative integer $N_0$ where for $0\leq k\leq N_0$ there exists $c_k^{(1)}(x_2,x_3,...,x_s)\in \C[x_2,x_3,...,x_s]$ such that 
\begin{equation*}
    f(x_1,...,x_s)=\sum_{k=0}^{N_0}c_k^{(1)}(x_2,...,x_s)x_1^k.
\end{equation*}Fix allowed $a_2,...,a_s$. Then if $\textbf{v}_1:=(q^{a_2},...,q^{a_s})$, we have 
\begin{equation*}
    g_{\textbf{v}_1}(x_1):=f(x_1,q^{a_2},...,q^{a_s})\in \C[x_1].
\end{equation*}Note $\{q^{a_1}\mid a_1 \in \frac{1}{2}\Z_{\geq a_2}, \ a_1-a_2\in\Z\}$ is an infinite set of zeros for the single-variable polynomial $g_{\textbf{v}_1}(x_1)$, so $g_{\textbf{v}_1}=0$ in $\C[x_1]$. Since $\textbf{v}_1$ was arbitrary, this implies $c_k^{(1)}(q^{a_2},...,q^{a_s})=0$ for all $0\leq k\leq N_0$ and for all $\textbf{a}\in S_n$. If we show $c_k^{(1)}=0$ for all $0\leq k\leq N_0$, then we are done. 

There is a nonnegative integer $N_1$ where for $0\leq k'\leq N_1$ there exists $c_{k'}^{(2)}(x_3,...,x_s)\in \C[x_3,...,x_s]$ such that 
\begin{equation*}
    c_{k}^{(1)}(x_2,...,x_s)=\sum_{k'=0}^{N_1}c_{k'}^{(2)}(x_3,...,x_s)x_2^{k'}.
\end{equation*}Fix allowed $a_3,...,a_s$. Then if $\textbf{v}_2:=(q^{a_3},...,q^{a_s})$, we have 
\begin{equation*}
    g_{\textbf{v}_2}(x_2):=c_k^{(1)}(x_2,q^{a_3},...,q^{a_s})\in \C[x_2].
\end{equation*}Note $\{q^{a_2}\mid a_2 \in \frac{1}{2}\Z_{\geq a_3}, \ a_2-a_3\in\Z\}$ is an infinite set of zeros for the single-variable polynomial $g_{\textbf{v}_2}(x_2)$, $g_{\textbf{v}_2}=0$ in $\C[x_2]$. Since $\textbf{v}_2$ was arbitrary, this implies $c_{k'}^{(2)}(q^{a_3},...,q^{a_s})=0$ for all $0\leq k'\leq N_1$ and for all $\textbf{a}\in S_n$. Now if we show $c_{k'}^{(2)}=0$ for all $0\leq k'\leq N_1$, then we are done. 

It is clear that we may continue this process until we obtain the result $c_{\hat{k}}^{(s-1)}(q^{a_s})=0$ for all $0\leq \hat{k}\leq N_{s-2}$ and for all $\textbf{a}\in S_n$. This provides us with an infinite set of zeros for a single-variable polynomial, hence $c_{\hat{k}}^{(s-1)}=0$ for all $0\leq \hat{k}\leq N_{s-2}$ and we are done.
\end{proof}

\section{Generic Gelfand-Tsetlin Representations of $\uq$}\label{uq}
In this section, our goal is to construct and study generic Gelfand-Tsetlin representations of $\uq$. These are infinite-dimensional analogs to the representations whose formulas are given by \eqref{i213} and \eqref{i323}. We use Lemma \ref{zar3} to accomplish this, but this requires that we rationalize our formulas first. We achieve an invertible re-scaling of the basis so that the coefficients from the formulas are rational functions in $q^{\ell}$ and $q^m$. While the finite-dimensional representations were characterized by the number $\ell$, these generic Gelfand-Tsetlin representations are characterized by numbers $\ell$ and $m_0$. We impose conditions on $\ell$ and $m_0$ to avoid singularities in the formulas. We then impose an additional condition on $m_0$ so that we may reasonably study the length of these new representations. We provide an example of an irreducible representation and one of maximum length. Lastly, we obtain a Casimir element of $\uq$ and its eigenvalue on basis vectors. This is useful for studying the length of generic Gelfand-Tsetlin representations of $\uqq$ constructed in Section \ref{uqq}.

\subsection{Re-Scaling Basis}
Recall the finite-dimensional representations $V_{\ell}$ of $\uq$ given in \cite[Proposition~1]{ga}, with action provided in \eqref{i213} and \eqref{i323}. In order to use Lemma \ref{zar3} to prove the existence of our generic Gelfand-Tsetlin representations, we need to rationalize the formulas from $V_{\ell}$. We achieve this via an invertible re-scaling of the basis. A rationalized Gelfand-Tsetlin basis for the quantum group $U_q(\mathfrak{gl}_n)$ was also constructed in \cite[Theorem~2.11]{kim}. 
\begin{lemma}\label{lemscale3}
There exists a re-scaling of the basis vectors of the finite-dimensional $\uq$-representation $V_{\ell}$ such that the coefficients in the relations are rational functions in $q^{\ell}$ and $q^{m}$.
\end{lemma}
\begin{proof}
For each $m\in\{-\ell,\ell+1,...,\ell-1,\ell\}$, we wish to re-scale the orthonormal basis vectors in the representation $V_{\ell}$ of $\uq$ so that the relations for the action of $\uq$ involve only rational functions of $q$-numbers, i.e. we eliminate square roots. Say $V_{\ell}=\operatorname{span}_{\C}\{\ket{m}':\ket{m}'=\mu_m\cdot \ket{m}, \mu_m\in \C^{\times}\}$. So the task is to find $\mu_m$ for each $m$ so that we eliminate square roots from the relations.

We define each $\mu_m$ recursively. Set $\mu_{\ell}=1$. Then for all $m<\ell$, define $\mu_m:=\mu_{m+1}\cdot A_{\ell,m}$ where $A_{\ell,m}$ is from \eqref{alm}. Note that
\begin{equation*}
    \lim_{m\to 0}\dfrac{[m]}{[2m]}=\dfrac{1}{2},
\end{equation*}so if $m=0$ (resp. $m=-1$), then we set $\frac{[m]}{[2m]}=\frac{1}{2}$ (resp. $\frac{[m+1]}{[2m+2]}=\frac{1}{2}$). This ensures that $A_{\ell,m}$ is always defined.

Suppose $A_{\ell,m}=0$. This can only happen when $[\ell+m+1]=0$ or $[\ell-m]=0$. By Lemma \ref{lemq}\eqref{lemq1}, this means $m=-\ell-1$ or $m=\ell$. But $m$ is never allowed to be $-\ell-1$, and we only use $A_{\ell,m}$ for scaling when $m<\ell$. Therefore our re-scaling of the basis vectors is well-defined since none of the $\mu_m$ are zero.

Then
\begin{equation*}
    I_{21}.\ket{m}'=\mu_m \cdot I_{21}.\ket{m}=\mu_m\cdot i[m]\cdot \ket{m}=i[m]\cdot \ket{m}'
\end{equation*}and
\begin{align*}
    I_{32}.\ket{m}'&=\mu_m\cdot I_{32}.\ket{m}\\
    &=\mu_m\cdot (A_{\ell,m}\ket{m+1}-A_{\ell,m-1}\ket{m-1})\\
    &= \dfrac{\mu_m}{\mu_{m+1}}\cdot A_{\ell,m}\ket{m+1}'-\dfrac{\mu_m}{\mu_{m-1}}\cdot A_{\ell,m-1}\ket{m-1}'\\
    &= \dfrac{\mu_{m+1}\cdot A_{\ell,m}}{\mu_{m+1}}\cdot A_{\ell,m}\ket{m+1}'-\dfrac{\mu_m}{\mu_{m}\cdot A_{\ell,m-1}}\cdot A_{\ell,m-1}\ket{m-1}'\\
    &= A_{\ell,m}^2\ket{m+1}'-\ket{m-1}'.
\end{align*}
\end{proof}

\subsection{Existence}
With our newly rationalized bases, we are prepared to define and prove the existence of generic Gelfand-Tsetlin representations of $\uq$. These now are characterized by two numbers $\ell$ and $m_0$. The variables appearing in the formulas no longer satisfy the conditions \eqref{highest_wt_odd}, \eqref{interlacing_odd}, and \eqref{interlacing_even}, which is why we need $m_0$. This number provides a sort of ``anchor" for our basis, since every basis vector is characterized by a number $m$, an integer shift of $m_0$. We still need $\ell$ to characterize our representations since it appears in the formulas.

Let $\ell,m_0\in \C$ where $q^{2m}\neq -1$ for all $m\in m_0+\Z$. (See Lemma \ref{lemq}\eqref{lemq3}.) Then we define $V_{m_0}^{\ell}=\bigoplus_{m\in m_0+\Z}\C\cdot\ket{m}$ with the action
\begin{subequations}
\begin{equation} \label{i213generic}
    I_{21}.\ket{m}=i[m]\cdot\ket{m}
\end{equation}
\begin{equation} \label{i323generic}
    I_{32}.\ket{m}=\bigg(\dfrac{[m][m+1]}{[2m][2m+2]}\cdot [\ell+m+1][\ell-m]\bigg)\ket{m+1}-\ket{m-1}
\end{equation}
\end{subequations}
where $\frac{[0]}{[0]}:=\frac{1}{2}$ in the $I_{32}$-action.
\begin{theorem} \label{exist3}
$V_{m_0}^{\ell}$ with action from \eqref{i213generic} and \eqref{i323generic} is indeed a representation of $\uq$.
\end{theorem}
\begin{proof}
We follow the argument of \cite[Theorem~2]{maz}. Let $u=0$ be a relation in $\uq$ (see \eqref{rel1}--\eqref{rel3}). We want to show this relation holds in $V_{m_0}^{\ell}$, i.e. $u.\ket{m}=0$ for any $m\in m_0+\Z$. Using the formulas we can write 
\begin{equation*}
    u.\ket{m}=\sum_{k=-2}^2f_{\ket{m+k}}(q^{\ell},q^m)\ket{m+k}
\end{equation*}where $f_{\ket{m+k}}$ is a rational function in $q^{\ell},q^{m}$. We want to show $f_{\ket{m+k}}$ is identically zero. 

By Lemma \ref{lemscale3}, for the finite-dimensional representation $V_{\ell'}$ with highest weight $\ell'$ (where $\ell'\in\frac{1}{2}\Z_{\geq 0}$), there exists a basis $\{\ket{m'}:|m'|\leq \ell',\ \ell'- m' \in \Z\}$ such that the action of the generators of $\uq$ on a basis vector is given by \eqref{i213generic} and \eqref{i323generic}, except we replace $\ell$ with $\ell'$ and $m$ with $m'$. By \cite[Proposition~1]{ga} and Lemma \ref{lemscale3}, $f_{\ket{m+k}}$ is zero when evaluated at any $q^{\ell'}$ and $q^{m'}$ that define a Gelfand-Tsetlin pattern for $\uq$. Therefore by Lemma \ref{zar3}, $f_{\ket{m+k}}=0$ and we are done.
\end{proof}

\subsection{Irreducibility and Finite Length}
Adding a natural condition on $m_0$ for convenience, we provide sufficient conditions for irreducibility of generic Gelfand-Tsetlin representations in Theorem \ref{irred3}, and an upper bound for the length in Theorem \ref{finlength3}. In the following section, we use these theorems to provide particular values of $\ell$ and $m_0$ that demonstrate examples of irreducibility and maximum length.
\begin{theorem}\label{irred3}
Let $m_0$ be such that $q^{2m_0+k}\neq -1$ for all $k\in \Z$. Let $\ell$ be such that $[\ell+m+1]\neq 0$ and $[\ell-m]\neq 0$ for all $m\in m_0+\Z$. Then $V_{m_0}^{\ell}$ is irreducible.
\end{theorem}
\begin{proof}
Suppose $U$ is a nonzero subrepresentation. Note $V_{m_0}^{\ell}=\bigoplus_{m\in m_0+\Z}\C\cdot \ket{m}$, where each $\ket{m}$ is a weight vector with respect to $I_{21}$. Since $0\neq U \subseteq V_{m_0}^{\ell}$, there exist nonzero $a_1,...,a_n$ such that $v:=a_1\ket{m_1}+\dots + a_n\ket{m_n}\in U$, where $m_j\neq m_k$ for all $j\neq k$. For simplicity, say $v_j:=a_j\ket{m_j}$ for all $1\leq j\leq n$. Note $v_j$ is a weight vector with respect to $I_{21}$ for all $j$, with the same weight as $\ket{m_j}$, $i[m_j]=:\mu_j$. Then
\begin{align*}
    U\owns v &= 1\cdot v_1+\dots+1\cdot v_n\\
    U\owns I_{21}.v &=\mu_1\cdot v_1+\dots+\mu_n\cdot v_n\\
    U\owns I_{21}^2.v &=\mu_1^2\cdot v_1+\dots+\mu_n^2\cdot v_n\\
    & \hspace{1.15cm}\vdots\\
    U\owns I_{21}^{n-1}.v &=\mu_1^{n-1}\cdot v_1+\dots+\mu_n^{n-1}\cdot v_n.
\end{align*}In other words,
\begin{equation*}
    \begin{pmatrix}
    1 & \dots & 1\\
    \mu_1 & \dots & \mu_n\\
    \mu_1^2 & \dots & \mu_n^2\\
    \vdots & & \vdots\\
    \mu_1^{n-1} & \dots & \mu_n^{n-1}
    \end{pmatrix} 
    \begin{pmatrix}
    v_1\\
    v_2\\
    \vdots\\
    v_n
    \end{pmatrix}=\begin{pmatrix}
    v\\
    I_{21}.v\\
    \vdots\\
    I_{21}^{n-1}.v
    \end{pmatrix}\in U^n.
\end{equation*}By Lemma \ref{lemq}\eqref{lemq4}, $\mu_j\neq \mu_k$ when $j\neq k$. Then the matrix is an invertible Vandermonde matrix. This means that each $v_j$ is a linear combination of elements of $U$, hence $v_j\in U$ for all $1\leq j\leq n$. Then $U\owns \frac{1}{a_j}v_j=\ket{m_j}$ for all $1\leq j\leq n$.

Note  
\begin{equation*}
    U\owns I_{32}.\ket{m}=\bigg(\dfrac{[m][m+1]}{[2m][2m+2]}\cdot [\ell+m+1][\ell-m]\bigg)\ket{m+1}-\ket{m-1},
\end{equation*}and by our choice of admissible $\ell$ and $m_0$, neither coefficient can be zero. (Note the equality in the proof for Lemma \ref{lemq}\eqref{lemq3} shows that $\frac{[m]}{[2m]}$ is never 0.) Then by the above argument, it follows that $\ket{m}\in U$ for all $m\in m_0+\Z$, hence $U=V_{m_0}^{\ell}$.
\end{proof}

Now in Theorem \ref{finlength3} we find an upper bound for the length of generic Gelfand-Tsetlin representations, and provide an explicit composition series in the proof.
\begin{theorem}\label{finlength3}
Let $m_0$ be such that $q^{2m_0+k}\neq -1$ for all $k\in \Z$. Then $V_{m_0}^{\ell}$ has length of at most $3$.
\end{theorem}
\begin{proof}
Let $S:=\{m \in m_0+\Z\mid [\ell+m+1]=0 \ \text{or} \ [\ell-m]=0\}$. By Lemma \ref{lemq}\eqref{lemq2}, there is at most one $m_1\in m_0+\Z$ such that $[\ell+m_1+1]=0$, and there is at most one $m_2\in m_0+\Z$ such that $[\ell-m_2]=0$. Therefore $|S|\leq 2$. If $S$ is empty, then by Theorem \ref{irred3} we see that $V_{m_0}^{\ell}$ is irreducible.

Now suppose $m\in S$. Then $I_{32}.\ket{m}=-\ket{m-1}$. In the proof of Theorem \ref{irred3}, we see that our condition on $m_0$ grants us the following: if $U$ is a nonzero subrepresentation and $U\owns v=a_1\ket{m_1}+\dots+a_n\ket{m_n}$, then $\ket{m_j}\in U$ for all $1\leq j\leq n$. This implies then that $\big(\bigoplus_{k=0}^{\infty}\C\ket{m-k}\big)\subseteq \uq\ket{m}$. The other inclusion is clear from the actions of $I_{21}$ and $I_{32}$. Therefore $\uq\ket{m}=\bigoplus_{k=0}^{\infty}\C\ket{m-k}$. If $|S|=1$, then by the argument in Theorem \ref{irred3} it follows that $\uq\ket{m}$ is irreducible and so is $V_{m_0}^{\ell}/(\uq\ket{m})$.

Suppose $|S|=2$. Say $m_1,m_2\in S$ where $m_1=m_0+k_1$ and $m_2=m_0+k_2$ and $k_1> k_2$. Then a similar argument shows that we have composition series
\begin{equation*}
    0 \subset \uq\ket{m_2} \subset \uq\ket{m_1} \subset V_{m_0}^{\ell}.
\end{equation*}
\end{proof}

\subsection{Examples}
In Example \ref{exirred3} we provide an irreducible generic Gelfand-Tsetlin representation of $\uq$. Then in Example \ref{exmaxlength3} we provide a generic Gelfand-Tsetlin representation of maximum length.
\begin{example}\label{exirred3}
    Let $\ell=\frac{1}{2}$, $m_0=0$. Then $V_{m_0}^{\ell}$ is irreducible.  
    \end{example}
    Since $q$ is not a root of unity, $q^{2m_0+k}=q^k\neq -1$ for all $k\in \Z$. Then for all $k\in \Z$, $[\ell+(m_0+k)+1]=[\frac{1}{2}+k+1]\neq 0$ and $[\ell-(m_0+k)]=[\frac{1}{2}-k]\neq 0$ by Lemma \ref{lemq}\eqref{lemq1}. Therefore $V_{m_0}^{\ell}$ is irreducible.
    \begin{example}\label{exmaxlength3}
    Let $\ell=1$, $m_0=0$. Then $V_{m_0}^{\ell}$ has length 3. 
    \end{example}
    $[\ell+(m_0-2)+1]=[1-2+1]=0$ and $[\ell-(m_0+1)]=[1-1]=0$. Therefore the length of $V_{m_0}^{\ell}$ is $3$. Its composition series is
    \begin{equation*}
        \{0\}\subset \uq \ket{-2}\subset\uq\ket{1}\subset V_{m_0}^{\ell}
    \end{equation*}where $\uq\ket{-2}=\bigoplus_{k=0}^{\infty}\C\ket{-2-k}$ and $\uq\ket{1}=\bigoplus_{k=0}^{\infty}\C\ket{1-k}$.
    
\subsection{Casimir Element}\label{cas}
In order to work towards irreducibility for $\uqq$, we need a Casimir element in $\uq$. This is given in \cite[Example~4.11]{mol}, but they are using a different presentation of $\uq$ from us. Their presentation is as follows: 
we have generators $s_{21}$, $s_{32}$, and $s_{31}$ satisfying
\begin{subequations}
\begin{align}
    [s_{21},s_{32}]_q&=-(q-q\inv)s_{31}, \label{srel1}\\
    [s_{32},s_{31}]_q&=-(q-q\inv)s_{21}, \label{srel2}\\
    [s_{31},s_{21}]_q&=-(q-q\inv)s_{32}, \label{srel3}
\end{align}
\end{subequations}
where $[a,b]_q=ab-qba$.

This same presentation of $\uq$ is given in \cite[Section~2.16.9]{mo}, and the connection between the presentations given by \eqref{srel1}--\eqref{srel3} and scaled versions of \eqref{rel1}--\eqref{rel3} are also provided. Thus we make the following identifications:
\begin{subequations}
\begin{align}
    s_{21}&=q^{-1/2}(q-q\inv)I_{21},\label{sid1} \\
    s_{32}&=q^{-1/2}(q-q\inv)I_{32},\label{sid2} \\
    s_{31}&=-q\inv (q-q\inv)I_{31} \label{sid3}
\end{align}
\end{subequations}
where $I_{31}:=[I_{21},I_{32}]_q$. It is straightforward to check \eqref{srel1}--\eqref{srel3} are still satisfied under these identifications.

Now we obtain the following proposition.
\begin{proposition}\label{cq}
The following are true.
\begin{enumerate}[{\rm (i)}]
    \item\label{cqcentral} $C_q:=-I_{21}^2-q\inv I_{31}^2-q^2I_{32}^2-(q-q\inv)I_{21}I_{32}I_{31}$ is central in $\uq$.
    \item\label{cqeigval} $C_q$ has an eigenvalue of $[\ell]^2+q^{\ell+1}[\ell]$ on a basis vector $\ket{m}\in V_{m_0}^{\ell}$.
\end{enumerate}
\end{proposition}
\begin{proof}
According to \cite[Example~4.11]{mol}, the following is central in $\uq$:
\begin{equation*}
    -q^2s_{21}^2-q^2s_{31}^2-q^4s_{32}^2+q^3s_{21}s_{32}s_{31}+q^3+2q.
\end{equation*}So by applying the identifications \eqref{sid1}--\eqref{sid3} to this, we get 
\begin{equation*}
    -q^2s_{21}^2-q^2s_{31}^2-q^4s_{32}^2+q^3s_{21}s_{32}s_{31}+q^3+2q
\end{equation*}
\begin{equation*}
    =q(q-q\inv)^2(-I_{21}^2-q\inv I_{31}^2-q^2I_{32}^2-(q-q\inv)I_{21}I_{32}I_{31})+q^3+2q.
\end{equation*}For convenience, we remove the $q(q-q\inv)^2$ prefactor and the $q^3+2q$ terms from the central element. This concludes the proof of \eqref{cqcentral}.

By Schur's Lemma, this element acts as a scalar on an irreducible representation $V_{\ell}$ (see \eqref{i213} and \eqref{i323}). By Lemma \ref{zar3}, it also acts by the same scalar on $V_{m_0}^{\ell}$. We find the eigenvalue by looking at the $\ket{m}$ coefficient. The coefficient is then
\begin{equation*}
    [m]^2+(-2q[m+1][m]+q^2[m]^2+[m+1]^2+q^2)A_{\ell,m}^2+(-2q[m][m-1]+q^2[m]^2+[m-1]^2+q^2)A_{\ell,m-1}^2
\end{equation*}where $A_{\ell,m}$ is from \eqref{alm}. We use the formula $[a+b]=q^a[b]+q^{-b}[a]$ to write all $q$-numbers in terms of $[m]$. Then the coefficient simplifies to
\begin{equation*}
    [m]^2+(q^{-2m}+q^2)A_{\ell,m}^2+(q^{2m}+q^2)A_{\ell,m-1}^2.
\end{equation*}Since our element is central, we find the eigenvalue by simply setting $m=0$. We get
\begin{align*}
    (1+q^2)A_{\ell,0}^2+(1+q^2)A_{\ell,-1}^2&=(1+q^2)\bigg(\frac{1}{2}\cdot \frac{1}{q+q\inv}[\ell][\ell+1]+\frac{1}{q\inv+q}\cdot\frac{1}{2}[\ell+1][\ell]\bigg)\\
    &=\frac{1}{2}q(2[\ell][\ell+1])\\
    &=q[\ell](q^{\ell}+q\inv[\ell])\\
    &=[\ell]^2+q^{\ell+1}[\ell].
\end{align*}
\end{proof}

\section{Generic Gelfand-Tsetlin Representations of $\uqq$}\label{uqq}
In this final section, our goal is to construct and study generic Gelfand-Tsetlin representations of $\uqq$. These are infinite-dimensional analogs to the representations whose formulas are given by \eqref{i21}--\eqref{i43}. We use Lemma \ref{zar4} to accomplish this, but this requires that we rationalize our formulas first. We achieve an invertible re-scaling of the basis so that the coefficients from the formulas are rational functions in $q^p$, $q^r$, $q^{\ell}$, and $q^m$. While the finite-dimensional representations were characterized by numbers $p$ and $r$, these generic Gelfand-Tsetlin representations are characterized by numbers $p$, $r$, $\ell_0$, and $m_0$. We impose conditions on these numbers to avoid singularities in the formulas. We then impose an additional condition on $\ell_0$ and $m_0$ so that we may reasonably study the length of these new representations. We provide an example of an irreducible representation and one of maximum length. 

\subsection{Re-Scaling Basis}
In order to define and prove the existence of our generic Gelfand-Tsetlin representations of $\uqq$, we rationalize the formulas from \eqref{i21}--\eqref{clm} so that we may then apply Lemma \ref{zar4}. Like in Lemma \ref{lemscale3}, we achieve this through an invertible re-scaling of the basis.
\begin{lemma}\label{lemscale4}
There exists a re-scaling of the basis vectors of the finite-dimensional $\uqq$-representation $V_{p,r}$ such that the coefficients in the relations are rational functions of $q^{p},q^{r},q^{\ell},q^{m}$.
\end{lemma}
\begin{proof}
Let $p\in \frac{1}{2}\Z_{\geq0}$, $r\in \frac{1}{2}\Z$, $p-r\in \Z$, and $p\geq |r|$. Let $V_{p,r}$ be the corresponding finite-dimensional irreducible representation of $\uqq$. Given the basis $\{\ket{\ell,m}\mid p\geq \ell \geq |r|, \ \ell \geq m \geq -\ell, \ p-\ell\in\Z, \ \ell-m\in\Z\}$ from \cite[Proposition~2]{ga}, we wish to form a new basis $\{\ket{\ell,m}'\}$ where $\ket{\ell,m}'=\mu_{\ell,m}\cdot \ket{\ell,m}$ so that the coefficients in the relations for the representation are rational functions of $q^{p},q^{r},q^{\ell},q^{m}$. 

We choose $\mu_{p,p}=1$, $\mu_{\ell,m}=A_{\ell,m}\cdot \mu_{\ell,m+1}$ when $m<\ell$, and $\mu_{\ell,\ell}=B_{\ell,\ell}\cdot\mu_{\ell+1,\ell}$ when $\ell<p$. $A_{\ell,m}$ and $B_{\ell,m}$ are from \eqref{alm} and \eqref{blm}, respectively. Then when $p>\ell>m$,
\begin{align*}
    \mu_{\ell,m}&=A_{\ell,m} A_{\ell,m+1}\dots A_{\ell,\ell-1}\cdot \mu_{\ell,\ell}\\
    &= A_{\ell,m} A_{\ell,m+1}\dots A_{\ell,\ell-1}B_{\ell,\ell}\cdot\mu_{\ell+1,\ell}\\
    &= A_{\ell,m} A_{\ell,m+1}\dots A_{\ell,\ell-1}B_{\ell,\ell}A_{\ell+1,\ell}\cdot\mu_{\ell+1,\ell+1}\\
    &= \bigg(\prod_{k=m}^{\ell-1}A_{\ell,k}\bigg)\bigg(\prod_{j=\ell}^{p-1}B_{j,j}A_{j+1,j}\bigg).
\end{align*}In the $\uq$ case, we saw $A_{\ell,k}\neq 0$ when $k<\ell$. Now we have to check $B_{j,j}\neq 0$ when $j<p$. By Lemma \ref{lemq}\eqref{lemq1}, it suffices to check that $a\neq 0$ for each $q$-number $[a]$ in the numerator. We know $p+j+2\neq 0$ since $j$ must be positive ($j\geq |r|$). Since $j<p$, $p-j\neq 0$. Suppose $j+r+1=0$. Then $r=-(j+1)$, so $j\geq|r|=j+1$, which is a contradiction. Likewise, if $j-r+1=0$ then we also get $j\geq j+1$. Obviously $j+j+1\neq 0$ since $j$ is nonnegative and $j-j+1\neq 0$. Therefore $\mu_{\ell,m}\neq 0$ when $p>\ell>m$. Note that when $\ell=m$, we just omit the first parenthetical product for $\mu_{\ell,m}$, so now we know $\mu_{\ell,m}\neq 0$ for $p>\ell\geq m$. If $p=\ell$, then $\mu_{\ell,m}$ is just $\mu_m$ from our re-scaling in the $\uq$ case, which we know to be nonzero. Thus we have an invertible re-scaling of the basis for $V_{p,r}$.

Now
\begin{subequations}
\begin{equation*}
    I_{21}.\ket{\ell,m}'=\mu_{\ell,m}\cdot i[m]\ket{\ell,m}=i[m]\ket{\ell,m}',
\end{equation*}
\begin{equation*}
     I_{32}.\ket{\ell,m}'=\mu_{\ell,m}(A_{\ell,m}\ket{\ell,m+1}-A_{\ell,m-1}\ket{\ell,m-1})
\end{equation*}
\begin{equation*}
    = \dfrac{\mu_{\ell,m}}{\mu_{\ell,m+1}}\cdot A_{\ell,m}\ket{\ell,m+1}'-\dfrac{\mu_{\ell,m}}{\mu_{\ell,m-1}}\cdot A_{\ell,m-1}\ket{\ell,m-1}',
\end{equation*}and likewise
\begin{equation*}
    I_{43}.\ket{\ell,m}'=\dfrac{\mu_{\ell,m}}{\mu_{\ell+1,m}}\cdot B_{\ell,m} \ket{\ell+1,m}'-\dfrac{\mu_{\ell,m}}{\mu_{\ell-1,m}}\cdot B_{\ell-1,m} \ket{\ell-1,m}'+iC_{\ell,m}\ket{\ell,m}'
\end{equation*}
\end{subequations}where $C_{\ell,m}$ comes from \eqref{clm}.
From our work in the $\uq$ case, we know 
\begin{equation*}
     I_{32}.\ket{\ell,m}'=A_{\ell,m}^2\ket{\ell,m+1}'-\ket{\ell,m-1}'.
\end{equation*}So it remains to check that the coefficients in $I_{43}.\ket{\ell,m}'$ are rational functions in $q^{p},q^{r},q^{\ell},q^{m}$. $C_{\ell,m}$ already has this form, so we check the other terms.
\begin{align*}
    \dfrac{\mu_{\ell,m}}{\mu_{\ell+1,m}}&=\dfrac{\bigg(\prod_{k=m}^{\ell-1}A_{\ell,k}\bigg)\bigg(\prod_{j=\ell}^{p-1}B_{j,j}A_{j+1,j}\bigg)}{\bigg(\prod_{k=m}^{\ell}A_{\ell+1,k}\bigg)\bigg(\prod_{j=\ell+1}^{p-1}B_{j,j}A_{j+1,j}\bigg)}\\
    &= \dfrac{\prod_{k=m}^{\ell-1}A_{\ell,k}}{\prod_{k=m}^{\ell}A_{\ell+1,k}}\cdot B_{\ell,\ell}A_{\ell+1,\ell}\\
    &=\bigg(\prod_{k=m}^{\ell-1}\dfrac{A_{\ell,k}}{A_{\ell+1,k}}\bigg)B_{\ell,\ell}\\
    &=\bigg(\prod_{k=m}^{\ell-1}\dfrac{[\ell+k+1][\ell-k]}{[\ell+k+2][\ell-k+1]}\bigg)^{1/2}B_{\ell,\ell}\\
    &=\bigg(\dfrac{[\ell+m+1][1]}{[2\ell+1][\ell-m+1]}\bigg)^{1/2}B_{\ell,\ell}\\
    &=\bigg(\dfrac{[\ell+m+1][1]}{[2\ell+1][\ell-m+1]}\bigg)^{1/2}B_{\ell,m}\bigg(\dfrac{[2\ell+1][1]}{[\ell+m+1][\ell-m+1]}\bigg)^{1/2}\\
    &=\dfrac{B_{\ell,m}}{[\ell-m+1]}.
\end{align*}So now we know $\mu_{\ell,m}=\frac{B_{\ell,m}}{[\ell-m+1]}\cdot \mu_{\ell+1,m}$, which means
\begin{equation*}
    \mu_{\ell-1,m}=\dfrac{B_{\ell-1,m}}{[\ell-m]}\cdot \mu_{\ell,m},
\end{equation*}hence
\begin{equation*}
    \dfrac{\mu_{\ell,m}}{\mu_{\ell-1,m}}=\dfrac{[\ell-m]}{B_{\ell-1,m}}.
\end{equation*}
\end{proof}

\subsection{Existence}
With the newly rationalized formulas, we are prepared to define and prove the existence of our generic Gelfand-Tsetlin representations of $\uqq$. These now are characterized by the numbers $p$, $r$, $\ell_0$, and $m_0$. The variables appearing in the formulas no longer satisfy the conditions \eqref{highest_wt_even}, \eqref{interlacing_odd}, and \eqref{interlacing_even}, which is why we need $\ell_0$ and $m_0$. These numbers provide ``anchors" for our basis, since every basis vector is characterized by numbers $\ell$ and $m$, integer shifts of $\ell_0$ and $m_0$, respectively. We still need $p$ and $r$ to characterize our representations since they appear in the formulas.

Let $p,r,\ell_0,m_0\in \C$ where $q^{2m}\neq -1$ for all $m\in m_0+\Z$, $q^{2\ell}\neq 1$ and $q^{4\ell+2}\neq 1$ for all $\ell\in \ell_0+\Z$. Then we define 
\begin{equation*}
    V_{\ell_0,m_0}^{p,r}=\bigoplus_{\substack{\ell\in \ell_0+\Z\\m\in m_0+\Z}}\C\cdot \ket{\ell,m}
\end{equation*}
and define action
\begin{subequations}
\begin{equation}
\label{myeq1a}
    I_{21}.\ket{\ell,m}=i[m]\ket{\ell,m},
\end{equation}
\begin{equation}
\label{myeq1b}
    I_{32}.\ket{\ell,m}=a_{\ell,m}\ket{\ell,m+1}-\ket{\ell,m-1},
\end{equation}
\begin{equation}
\label{myeq1c}
    I_{43}.\ket{\ell,m}=b_{\ell,m}\ket{\ell+1,m}-[\ell-m]\ket{\ell-1,m}+ic_{\ell,m}\ket{\ell,m}
\end{equation}
\end{subequations}where
\begin{subequations}
\begin{equation}
\label{myeq2a}
    a_{\ell,m}=\dfrac{[m][m+1]}{[2m][2m+2]}\cdot [\ell+m+1][\ell-m],
\end{equation}
\begin{equation}
\label{myeq2b}
    b_{\ell,m}=\dfrac{[p+\ell+2][p-\ell][\ell+r+1][\ell-r+1][\ell+m+1]}{[\ell+1]^2[2\ell+1][2\ell+3]},
\end{equation}
\begin{equation}
\label{myeq2c}
    c_{\ell,m}=\dfrac{[p+1][r][m]}{[\ell+1][\ell]}.
\end{equation}
\end{subequations}
\begin{theorem} \label{exist4}
$V_{\ell_0,m_0}^{p,r}$ with the above action is a representation of $\uqq$.
\end{theorem}
\begin{proof}
Again, we follow the argument of \cite[Theorem~2]{maz}. Let $u=0$ be a relation in $\uqq$ (see \eqref{rel1}--\eqref{rel3}). We want to show this relation holds in $V_{\ell_0,m_0}^{p,r}$, i.e. $u.\ket{\ell,m}=0$ for any $\ell\in\ell_0+Z$ and $m\in m_0+\Z$. Using the formulas we can write 
\begin{equation*}
    u.\ket{\ell,m}=\sum_{k=-2}^2\sum_{k'=-2}^2f_{\ket{\ell+k,m+k'}}(q^p,q^r,q^{\ell},q^m)\ket{\ell+k,m+k'}
\end{equation*}where $f_{\ket{\ell+k,m+k'}}$ is a rational function in $q^{p},q^{r},q^{\ell},q^{m}$. We want to show $f_{\ket{\ell+k,m+k'}}$ is identically zero. 

By Lemma \ref{lemscale4}, for the finite-dimensional representation $V_{p',r'}$ with highest weight $p'\geq |r'|$ (where $(p',|r'|)\in \Z_{\geq 0}^2 \cup (\frac{1}{2}+\Z_{\geq 0})^2$) there exists a basis $\{\ket{\ell',m'}\mid p'\geq\ell'\geq|r'|; \ell'\geq|m'|;p'- r',p'- \ell', p'- m' \in \Z)\}$ such that the action of the generators of $\uqq$ on a basis vector is given by \eqref{myeq1a}--\eqref{myeq1c}, except we replace $p$ with $p'$, $r$ with $r'$, $\ell$ with $\ell'$, and $m$ with $m'$. By \cite[Proposition~2]{ga} and Lemma \ref{lemscale4}, $f_{\ket{\ell+k,m+k'}}$ is zero when evaluated at any $q^{p'},q^{r'},q^{\ell'},q^{m'}$ that define a Gelfand-Tsetlin pattern for $\uqq$. Therefore by Lemma \ref{zar4}, each $f_{\ket{\ell+k,m+k'}}$ is identically 0.
\end{proof}

\subsection{Irreducibility and Finite Length}
Adding natural condition on $\ell_0$ and $m_0$ for convenience, we provide an upper bound for the length of generic Gelfand-Tsetlin representations of $\uqq$ and sufficient conditions for irreducibility in Theorem \ref{finlength4}. In the following section, we use this theorem to provide particular values of $p$, $r$, $\ell_0$, and $m_0$ that demonstrate examples of irreducibility and maximum length.
\begin{theorem}\label{finlength4}
Let $p,r,\ell_0,m_0\in \C$ where $q^{2m_0+k}\neq -1$, $q^{2\ell_0+2k}\neq 1$, and $q^{4\ell_0+2k}\neq 1$ for all $k\in\Z$. Then $V_{\ell_0,m_0}^{p,r}$ has a length of at most $6$. Moreover, if for all $k\in\Z$ we have $q^{2p+2\ell_0+2k}\neq 1$, $q^{2p-2\ell_0+2k}\neq 1$, $q^{2\ell_0+2r+2k}\neq 1$, $q^{2\ell_0-2r+2k}\neq 1$, $q^{2\ell_0+2m_0+2k}\neq 1$, and $q^{2\ell_0-2m_0+2k}\neq 1$, then $V_{\ell_0,m_0}^{p,r}$ is irreducible.
\end{theorem}
\begin{proof}
Suppose $U$ is a subrepresentation, $S\subseteq \Z^2$, and $\sum_{(k_1,k_2)\in S}a_{k_1,k_2}\ket{\ell_0+k_1,m_0+k_2}\in U$, where $a_{k_1,k_2}\in \C$. Let ${S}_{k_2}:=\{k_1:(k_1,k_2)\in S\}$. Then as seen in the proof of Theorem \ref{irred3}, since $q^{2m_0+k}\neq -1$ for all $k\in\Z$, $\sum_{k_1\in{S}_{k_2}}a_{k_1,k_2}\ket{\ell_0+k_1,m_0+k_2}\in U$. By the same argument, since $q^{4\ell_0+2k}\neq 1$ for all $k\in\Z$ (see Lemma \ref{lemq}\eqref{lemq5} and Proposition \ref{cq}), it follows that $\ket{\ell_0+k_1,m_0+k_2}\in U$ for all $(k_1,k_2)\in S$. Therefore, by \eqref{myeq1a}-\eqref{myeq2c}, if the following $q$-numbers are never equal to 0, then $V_{\ell_0,m_0}^{p,r}$ is irreducible, since any nonzero subrepresentation would have to equal the entire space: $[p+\ell+2]$, $[p-\ell]$, $[\ell+r+1]$, $[\ell-r+1]$, $[\ell+m+1]$, and $[\ell-m]$, for all $\ell\in \ell_0+\Z$ and for all $m\in m_0+\Z$. Equivalently, if for all $k\in\Z$ we have $q^{2p+2\ell_0+2k}\neq 1$, $q^{2p-2\ell_0+2k}\neq 1$, $q^{2\ell_0+2r+2k}\neq 1$, $q^{2\ell_0-2r+2k}\neq 1$, $q^{2\ell_0+2m_0+2k}\neq 1$, and $q^{2\ell_0-2m_0+2k}\neq 1$, then $V_{\ell_0,m_0}^{p,r}$ is irreducible.

Let $R:=\{\ell\in \ell_0+\Z:[p+\ell+2]=0, \ [p-\ell]=0, \ [\ell+r+1]=0, \ \text{or} \ [\ell-r+1]=0\}$. Since $p$ and $r$ are fixed, Lemma \ref{lemq}\eqref{lemq2} ensures that $|R|\leq 4$. Suppose we have integers $k_1$ and $k_2$ such that $q^{2\ell_0+2r+2k_1}=1$ and $q^{2\ell_0-2r+2k_2}=1$. Then
\begin{equation*}
    q^{2\ell_0+2r+2k_1}\cdot q^{2\ell_0-2r+2k_2}=1
\end{equation*}which implies
\begin{equation*}
    q^{4\ell_0+2(k_1+k_2)}=1.
\end{equation*}This contradicts one of our initial conditions. Likewise, we cannot have both $[p+\ell+2]=0$ and $[p-\ell]=[\ell-p]=0$. Therefore $|R|\leq 2$. We name the potential elements of $R$ $\ell_1$ and $\ell_2$, where $\ell_1-\ell_2\geq 0$. By the same argument as shown above, it is impossible to have integers $k_1$ and $k_2$ such that $q^{2\ell_0+2m_0+2k_1}=1$ and $q^{2\ell_0-2m_0+2k_2}=1$. We therefore examine three distinct cases: one where $q^{2\ell_0+2m_0+2k}\neq 1$ and $q^{2\ell_0-2m_0+2k}\neq 1$ for all $k\in\Z$, one where there exists $k'\in\Z$ such that $q^{2\ell_0-2m_0+2k'}=1$ , and one where there exists $k'\in \Z$ such that $q^{2\ell_0+2m_0+2k'}=1$. In the following arguments, we will assume for convenience that we have distinct $\ell_1$ and $\ell_2$. The arguments in other cases are similar.
\begin{case}
$q^{2\ell_0+2m_0+2k}\neq 1$ and $q^{2\ell_0-2m_0+2k}\neq 1$ for all $k\in\Z$.
\end{case}\noindent This is analogous to what we saw in Theorem \ref{finlength3}. By the same arguments shown there, we obtain the following composition series:
\begin{equation*}
    0 \subset \uqq\ket{\ell_2,m_2} \subset \uqq\ket{\ell_1,m_1} \subset V_{\ell_0,m_0}^{p,r}
\end{equation*}for any $m_1,m_2\in m_0+\Z$.
\begin{case}
There exists $k'\in\Z$ such that $q^{2\ell_0-2m_0+2k'}=1$.
\end{case}\noindent Let $S:=\{\ket{\ell,m}:[\ell-m]=0\}$, and let $U$ be the subrepresentation generated by $S$. By Lemma \ref{lemq}\eqref{lemq2}, if we have a pair $(\ell,m)$ such that $[\ell-m]=0$, then $S=\{\ket{\ell+k,m+k}:k\in\Z\}$. $U$ has the following basis: $\{\ket{\ell',m'-k}:\ket{\ell',m'}\in S, \ k\in\Z_{\geq 0}\}$. This is because we are not allowed to raise the $m$-component from $\ket{\ell',m'}$ by \eqref{myeq1b} and \eqref{myeq2a}, and if we start from a generating vector with higher $\ell$, say $\ket{\ell'+k',m'+k'}$ for some positive integer $k'$, we cannot lower the $\ell$-component without first lowering the $m$-component by \eqref{myeq1b}, \eqref{myeq2a}, and \eqref{myeq1c}. Let $M_j:=\bigoplus_{k\in \Z_{\geq0}, \ m\in m_0+\Z}\C\ket{\ell_j-k,m}$, and let $U_j:=U+M_j$ for $j\in\{1,2\}$. The $U_j$ are subrepresentations since they are the sum of subrepresentations. $V_{\ell_0,m_0}^{p,r}/U_1$ has basis $\{\ket{\ell',m'+k}:\ell'-\ell_1>0, \ \ket{\ell',m'}\in S, \ k\in\Z_{>0}\}$. Since $\ell'-\ell_1>0$ and we have modded out by $U$, it is straightforward that $V_{\ell_0,m_0}^{p,r}/U_1$ is generated by any nonzero vector, hence it is irreducible.  Similarly, $U_1/U_2$ and $U_2/U$ are irreducible. Let $U_j'=U\cap M_j$ for $j\in\{1,2\}$. $U/U_1'$, $U_1'/U_2'$, and $U_2'$ are generated by any nonzero vector. We obtain the following composition series:
\begin{equation*}
    0\subset U_2'\subset U_1'\subset U \subset U_2 \subset U_1 \subset V_{\ell_0,m_0}^{p,r}.
\end{equation*}
\begin{case}
There exists $k'\in \Z$ such that $q^{2\ell_0+2m_0+2k'}=1$.
\end{case}\noindent Let $T:=\{\ket{\ell,m}:[\ell+m+1]=0\}$, and let $W$ be the subrepresentation generated by $T$. By Lemma \ref{lemq}\eqref{lemq2}, if we have a pair $(\ell,m)$ such that $[\ell+m+1]=0$, then $T=\{\ket{\ell+k,m-k}:k\in\Z\}$. $W$ has the following basis: $\{\ket{\ell',m'-k}:\ket{\ell',m'}\in T, \ k\in\Z_{\geq 0}\}$. This is because we are not allowed to raise the $m$-component from $\ket{\ell',m'}$ by \eqref{myeq1b} and \eqref{myeq2a}, and if we start from a generating vector with lower $\ell$, say $\ket{\ell'-k',m'+k'}$ for some positive integer $k'$, we cannot raise the $\ell$-component without first lowering the $m$-component by \eqref{myeq1b}, \eqref{myeq2a}, \eqref{myeq1c}, \eqref{myeq2b}. Let $W_j:=W+M_j$ for $j\in\{1,2\}$. For $j\in\{1,2\}$, let $W_j'=W\cap M_j$. As in the previous case, we have the following composition series:
\begin{equation*}
    0\subset W_2'\subset W_1'\subset W \subset W_2 \subset W_1 \subset V_{\ell_0,m_0}^{p,r}.
\end{equation*}
\end{proof}

\subsection{Examples}
In Example \ref{exirred4} we provide an irreducible generic Gelfand-Tsetlin representation of $\uqq$. Then in Example \ref{exmaxlength4} we provide a generic Gelfand-Tsetlin representation of maximum length.
\begin{example}\label{exirred4}
When $p=r=m_0=0$ and $\ell_0=\frac{1}{4}$, $V_{\ell_0,m_0}^{p,r}$ is irreducible.
\end{example}\noindent Using Lemma \ref{lemq}\eqref{lemq1}, a simple argument by contradiction shows that the initial hypotheses of Theorem \ref{finlength4} are satisfied, hence $V_{\ell_0,m_0}^{p,r}$ has a length of at most 6. Then since $p=r=m_0=0$, the other hypotheses in Theorem \ref{finlength4} are satisfied and $V_{\ell_0,m_0}^{p,r}$ is irreducible. 
\begin{example}\label{exmaxlength4}
When $p=r=\ell_0=m_0=\frac{1}{4}$, $V_{\ell_0,m_0}^{p,r}$ has length 6.
\end{example}\noindent By Example \ref{exirred4}, we already know $q^{2\ell_0+2k}\neq 1$ and $q^{4\ell_0+2k}\neq 1$ for all $k\in\Z$. Suppose there exists $k'\in\Z$ such that $q^{2m_0+k'}=q^{\frac{1}{2}+k'}=-1$. Then $q^{1+2k'}=1$, hence $[k'+\frac{1}{2}]=0$. By Lemma \ref{lemq}\eqref{lemq1}, this is impossible. Therefore the initial hypotheses of Theorem \ref{finlength4} are satisfied.

Recall the definitions of $\ell_1$ and $\ell_2$ from the proof of Theorem \ref{finlength4}. Since $[p-\ell_0]=[0]=0$ and $[(\ell_0-1)-r+1]=[0]=0$, we have $\ell_1=\ell_0=\frac{1}{4}$ and $\ell_2=\ell_0-1=-\frac{3}{4}$. We also have $[\ell_0-m_0]=[0]=0$, therefore this example follows Case 2 in the proof of Theorem \ref{finlength4}.

We illustrate the subrepresentations that appear in the composition series in Figure 1. Each point in the grid represents a basis vector, where the origin represents $\ket{\ell_0,m_0}$. $M_2$ consists of everything in the red region, $M_1$ consists of everything in the blue region, and $U$ consists of everything in the orange region.
\begin{figure}\label{fig:ex}
\centering
\begin{tikzpicture}
\coordinate (Origin)   at (0,0);
    \coordinate (XAxisMin) at (-4,0);
    \coordinate (XAxisMax) at (4,0);
    \coordinate (YAxisMin) at (0,-4);
    \coordinate (YAxisMax) at (0,4);
    \fill[orange, domain=-4:4, variable=\x]
    (-4,-4)
    -- plot ({\x},{\x})
    -- (4,-4)
    --cycle;
    \draw [thin, black,-latex] (XAxisMin) -- (XAxisMax) node [right] {$\ell$};
    \draw [thin, black,-latex] (YAxisMin) -- (YAxisMax) node [above] {$m$};
    \foreach \x in {-3,...,3}{
      \foreach \y in {-3,...,3}{
        \node[draw,circle,inner sep=1pt,fill] at (\x,\y) {};
      }
    }
    \draw[red, ultra thick] (-1,-4) -- (-1,4);
    \draw[blue, ultra thick] (0,-4) -- (0,4);
    \draw[orange, ultra thick] (-4,-4) -- (4,4);
    \draw[pattern=vertical lines, pattern color=blue] (-4,4) rectangle (0,-4);
    \draw[pattern=horizontal lines, pattern color=red] (-4,4) rectangle (-1,-4);
    \clip (-2,-2) rectangle (2,2);
\end{tikzpicture}
\caption{Subrepresentations in the composition series from Example \ref{exmaxlength4}}
\end{figure}
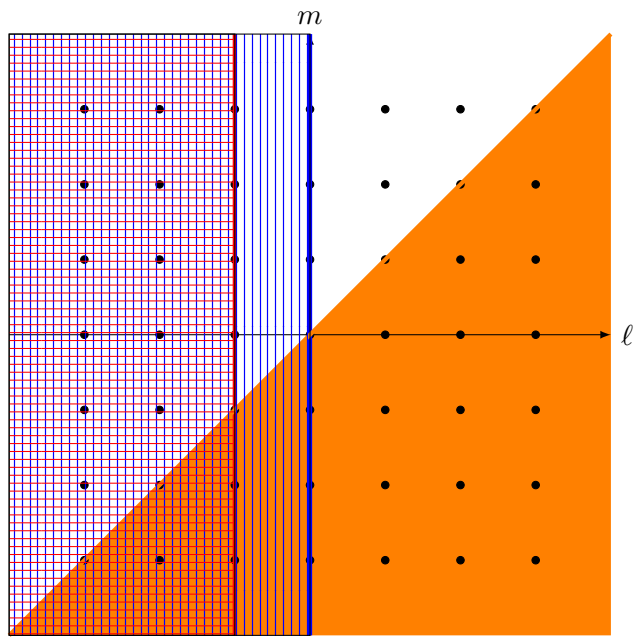

\end{document}